\documentclass[12pt,a4paper]{article}
\usepackage[dvipdfmx]{graphicx}
\usepackage{latexsym}
\usepackage{amssymb}
\usepackage{amsmath}
\usepackage{amscd}
\usepackage{amsthm}   %for springer style
\usepackage{amsfonts}
\usepackage{mathrsfs}
\usepackage{bm}
 \makeatletter
    
    \@addtoreset{equation}{section}
  \makeatother
\newtheorem{theorem}{Theorem}[section]

\newtheorem{lemma}[theorem]{Lemma}
\newtheorem{proposition}[theorem]{Proposition}
\newtheorem{fact}[theorem]{Fact}
\theoremstyle{definition}
\newtheorem{example}[theorem]{Example}

\newtheorem{remark}[theorem]{Remark}
\newcommand{\bx}{\mbox{\boldmath $x$}}

\newcommand{\be}{\mbox{\boldmath $e$}}
\newcommand{\e}{\mbox{\boldmath $e$}}

\newcommand{\bb}{\mbox{\boldmath $b$}}

\newcommand{\bnu}{\mbox{\boldmath $\nu$}}

\newcommand{\bs}{\mbox{\boldmath $s$}}
\newcommand{\bgamma}{\mbox{\boldmath $\gamma$}}

\newcommand{\bc}{\mbox{\boldmath $c$}}
\newcommand{\A}{{\cal A}}
\renewcommand{\phi}{\varphi}
\newcommand{\hess}{\operatorname{Hess}}
\newcommand{\inner}[2]{\left\langle{#1},{#2}\right\rangle}
\newcommand{\spann}[1]{\left\langle{#1}\right\rangle}
\newcommand{\ep}{\varepsilon}
\newcommand{\R}{{\mathbb R}}
\newcommand{\sgn}{{\operatorname{sgn}}}
\newcommand{\lon}{\longrightarrow}

\newcommand{\pmt}[1]{{\begin{pmatrix} #1  \end{pmatrix}}}
\title{\Large {\bf Flat surfaces along swallowtails}}
\author{Shyuichi IZUMIYA, Kentaro SAJI and Keisuke TERAMOTO}
\date{\today}
\begin{document}

\maketitle
\begin{abstract}
We consider developable surfaces along the singular set of a swallowtail
which are considered to be flat approximations of the swallowtail.
For the study of singularities of such developable surfaces, 
we introduce the notion of  Darboux frames along swallowtails and invariants.
As a by-product, we give a new example of a frontal which is locally homeomorphic to a swallowtail. 
\end{abstract}
\renewcommand{\thefootnote}{\fnsymbol{footnote}}
\footnote[0]{This work was supported by 
JSPS KAKENHI Grant Numbers 
JP26287009,JP26400087,JP17J02151.} 
\section{Introduction}
Recently, there appeared several articles concerning on 
differential geometry of singular surfaces in the 
Euclidean $3$-space \cite{FH1,FH2,MN,MS,MSUY,OT,SUY,suy3,WE}.
Wave fronts and frontals are particularly interesting singular surfaces 
which always have normal directions even along singularities.
Surfaces which have only
cuspidal edges and swallowtails as singularities are 
the generic wave fronts in the Euclidean $3$-space.
In this paper we consider a developable surfaces
along the singular locus of 
a swallowtail surface in the Euclidean $3$-space,
and a singular point of a frontal surface which has the
similar properties to a swallowtail.
Such a developable surface is called a {\it developable surface along\/} 
swallowtail, (or a singular point of a frontal surface
which have a
similar properties to a swallowtail).
Actually there are infinitely many developable surfaces along the singular locus of the swallowtail.
Since a frontal surface has the normal direction at any point 
(even at a singular point),
we focus on typical two developable surfaces along it.
One of them is a developable surface which is tangent to the swallowtail 
surface and another one is normal to it. 
These two developable surfaces are considered to be flat 
approximations of the swallowtail along the singular locus of it.
We investigate the singularities of these developable surfaces
and induce new invariants for the swallowtail.
For the purpose, we introduce the notion of 
Darboux frames along swallowtails which is analogous to the notion of 
Darboux frames along curves on regular surfaces (cf. \cite{HII,HI,IzuOta}).
Since the Darboux frame along a swallowtail is orthonormal frame, 
we can obtain the structure equation
and the invariants (cf. equation \eqref{eq:frenet}). 
We show that these invariants are related to the invariants 
which are known as basic invariants of the swallowtail 
in \cite{MS,MSUY,SUY}. 
By using the Darboux frame, we can directly and 
instinctively understand geometric properties of the swallowtail.
Moreover, if one of the three basic invariants is constantly equal to zero, 
we have special developable surfaces.

The similar investigation for cuspidal edges has been done in \cite{ist}.
This paper is not only a kind of continuous investigation of \cite{ist} but also
gives a new example of a frontal which is locally homeomorphic to a swallowtail as a by-product (cf. Example \ref{ex:ndfcyl}).
We only know a cupsidal crosscap as such an example so far as we know.

%This paper is a kind of continuation of \cite{ist}.
\section{Preliminaries}
\subsection{Preliminaries on frontals}
The precise definition of the swallowtail (surface) is given as follows:
The unit cotangent bundle $T^*_1\R^{3}$ of  $\R^{3}$ has the
canonical contact structure and can be identified with the unit
tangent bundle $T_1\R^{3}$. Let $\alpha$ denote the canonical
contact form on it. A map $i:M\to T_1\R^{3}$ is said to be {\it
isotropic\/} if the pull-back $i^*\alpha$ vanishes
identically. We call the image of $\pi\circ i$
the {\it wave front set\/} of $i$,
where $\pi:T_1\R^{3}\to\R^{3}$ is the canonical projection and we
denote it by\/ $W(i)$. Moreover, $i$ is called the {\it Legendrian
lift\/} of $W(i)$. 
With this framework, we define the notion of
fronts as follows: A map-germ $f:(\R^2,0) \to (\R^{3},0)$ is
called a {\it frontal\/} 
if there exists a unit vector field 
(called {\it unit normal of\/} $f$)
$\nu$ of $\R^{3}$ along $f$
such that
$L=(f,\nu):(\R^2,0)\to (T_1\R^{3},0)$ is
an isotropic map by an identification 
$T_1\R^3 = \R^3 \times S^2$, where $S^2$ is 
the unit sphere in $\R^3$ (cf. \cite{AGV}, see also \cite{krsuy}).
A frontal $f$ is a {\it front\/} if the above $L$ can be taken as
an immersion.
A point $q\in (\R^2, 0)$ is a singular point if $f$ is not an
immersion at $q$.
A map $f:M\to N$ 
between $2$-dimensional manifold $M$ and
$3$-dimensional manifold $N$ is called 
a frontal (respectively, a front)
if for any $p\in M$, the map-germ $f$ at $p$
is a frontal (respectively, a front).
A singular point $p$ of a map $f$ is called a {\it cuspidal edge\/}
if the map-germ $f$ at $p$ is $\mathcal{A}$-equivalent to
$(u,v)\mapsto(u,v^2,v^3)$ at $0$,
and a singular point $p$ is called a {\it swallowtail\/}
if the map-germ $f$ at $p$ is $\mathcal{A}$-equivalent to
$(u,v)\mapsto(u,4v^3+2uv,3v^4+uv^2)$ at $0$. 
(Two map-germs
$f_1,f_2:(\R^n,0)\to(\R^m,0)$ are $\mathcal{A}$-{\it
equivalent}\/ if there exist diffeomorphisms
$S:(\R^n,0)\to(\R^n,0)$ and $T:(\R^m,0)\to(\R^m,0)$ such
that $ f_2\circ S=T\circ f_1 $.) Therefore if the singular point $p$
of $f$ is a swallowtail, then $f$ at $p$ is a front.
Furthermore, cuspidal edges and swallowtails are 
two types of generic singularities of
fronts.
Let $f:(\R^2,0) \to (\R^{3},0)$ be a frontal and $\nu$ its
unit normal.
Let $\lambda$ be a function which is a non-zero functional multiplication
of the function
$$
\det(f_u,f_v,\nu)
$$
for some coordinate system $(u,v)$, and 
$(~)_u=\partial/\partial u$,
$(~)_v=\partial/\partial v$.
A singular point $p$ of $f$ is called {\it non-degenerate\/}
if $d\lambda(p)\ne0$.
Let $0$ be a non-degenerate singular point of $f$.
Then the set of singular points $S(f)$ is a regular curve,
we take a parameterization $\gamma(t)$ $(\gamma(0)=0)$ of it.
We set $\hat\bgamma=f\circ\gamma$ and call $\hat\bgamma$ 
the {\it singular locus}.
One can show that
there exists a vector field $\eta$ along $\gamma$,
such that 
$$
\ker df_{\gamma(t)}=\spann{\eta(t)}_{\R}.
$$
Set 
\begin{equation}\label{eq:criphi}
\phi(t)=\det(\gamma'(t),\eta(t)).
\end{equation}
Here, %and in what follows,
we denote 
%${}'=\partial /\partial u$ or
${}'=d /d t$.
A non-degenerate singular point $0$ is 
the {\it first kind\/} if
$\phi(0)\ne0$.
A non-degenerate singular point $0$ is 
the {\it second kind\/} if
$\phi(0)=0$ and $\phi'(0)\ne0$.
We remark that if $f$ is a front,
then the singular point of the first kind
is the cuspidal edge, and
the singular point of the second kind
is the swallowtail \cite{krsuy}.
The following criteria for cuspidal edge and swallowtail are
known.
\begin{fact}\label{fact:cri}
Let\/ $f:(\R^2,0)\to(\R^3,0)$ be a front,
and\/ $0$ a non-degenerate singularity.
Then the followings are equivalent:
\begin{itemize}
\item $0$ is cuspidal edge\/ $($respectively, swallowtail\/$)$,
\item $\phi(0)\ne0$ $($respectively, $\phi(0)=0$, $\phi'(0)\ne0)$,
\item $\eta\lambda(0,0)\ne0$ $($respectively, 
$\eta\lambda(0,0)=0$, $\eta\eta\lambda(0,0)\ne0)$.
\end{itemize}
\end{fact}

On the other hand, a developable surface is known to be a frontal, so that
the normal direction is well-defined at any point.
We say that a developable surface is an 
{\it osculating developable surface\/} along
$f$ if it contains the singular set of $f$ such that
the normal direction of the developable surface coincides 
with the normal direction of $f$ at any point of
the singular set.
We also say that a developable surface is a 
{\it normal developable surface\/} along $f$ if
it contains the singular set of $f$ such that 
the normal direction of the developable surface
belongs to the tangent plane of $f$ at any point of the singular set,
where the {\it tangent plane\/} of $f$ at $\gamma(t)$ is 
$\nu(\gamma(t))^\perp$.
In this paper, we study the geometric properties of a non-degenerate
singular point of a frontal $f$
using these two developable surfaces along $f$.
In particular, we show that the singular values of those 
developable surfaces characterize some geometric properties of $f$.

\subsection{Frames on non-degenerate singularities of frontals}
\label{sec:frame}
First, we show the following lemma to take a frame along a singular curve.
\begin{lemma}\label{lem:23type}
If\/ $0$ is a singular point of the second kind
of\/ $f$.
Then\/ 
$\hat\bgamma=f\circ \gamma$ satisfies\/ 
$\det(\hat\bgamma''(0),\hat\bgamma'''(0),\nu(0))\ne0$,
where\/ $\gamma$ is a parameterization of\/ $S(f)$ near\/ $0$.
In particular\/ $\hat\bgamma$ is the\/ $3/2$-cusp at\/ $0$.
\end{lemma}
Here, a {\it $3/2$-cusp\/} is a map-germ $(\R,0)\to(\R^3,0)$ 
which is $\A$-equivalent
to $t\mapsto(t^2,t^3,0)$ at $0$.
Taking a coordinate system $(u,v)$ satisfying
$S(f)=\{(u,v)\,|\,v=0\}$.
Then one can take a null vector field 
$\eta(u)=\partial_u+\ep(u)\partial_v$, $\ep(0)=0$, $\ep'(0)\ne0$. 
If we take such a coordinate system $(u,v)$, 
we set ${}'=\partial/\partial u$ or ${}'=d/du$ in what follows. 
Since our consideration in this paper is local,
we assume that 
\begin{equation}\label{eq:ep'}
\ep'(0)>0
\end{equation}
by changing the coordinate $(u,v)$ to $(u,-v)$ if necessary. 
Then $df(\eta)=0$ on the $u$-axis, so that 
there exists a vector valued function $g$
such that
\begin{equation}\label{eq:dfeta}
f_u(u,v)+\ep(u)f_v(u,v)=vg(u,v).
\end{equation}
Then differentiating \eqref{eq:dfeta},
we get
\begin{equation}\label{eq:dfetabibun}
\begin{array}{l}
f_{uu}(u,v)+\ep'(u)f_v(u,v)+\ep(u)f_{uv}(u,v)=vg_u(u,v),\\
f_{uv}(u,v)+\ep(u)f_{vv}(u,v)=g(u,v)+vg_u(u,v),\\
f_{uuu}(u,v)+\ep''(u)f_v(u,v)+2\ep' (u)f_{uv}(u,v)
+\ep(u) f_{uuv}(u,v)=vg_{uu}(u,v).
\end{array}
\end{equation}
On the $u$-axis, it holds that
\begin{equation}\label{eq:dfetabibunu}
\begin{array}{l}
f_{uu}(u,0)+\ep'(u)f_v(u,0)+\ep(u)f_{uv}(u,0)=0,\\
f_{uv}(u,0)+\ep(u)f_{vv}(u,0)=g(u,0),\\
f_{uuu}(u,0)+\ep''(u)f_v(u,0)+2\ep' (u)f_{uv}(u,0)
+\ep(u) f_{uuv}(u,0)=0.
\end{array}
\end{equation}
\begin{proof}[Proof of Lemma\/ {\rm \ref{lem:23type}}]
Since the assertion does not depend on the choice of
coordinate systems,
we take a coordinate system $(u,v)$ and 
take a null vector field as above.
Then 
$
\det(\hat\bgamma''(0),\hat\bgamma'''(0),\nu(0))
=
\det(f_{uu},f_{uuu},\nu)(0,0)
$.
On the other hand,
\begin{align*}
\eta\eta\lambda(0)=&
\det(\eta\eta f_u,f_v,\nu)(0)+
2\det(\eta f_u,\eta f_v,\nu)(0)+
2\det(\eta f_u,f_v,\eta\nu)(0)\\
=&
\det(f_{uuu}+\ep'f_{uv},f_v,\nu)(0)+
2\det(f_{uu},f_{uv},\nu)(0)+
2\det(f_{uu},f_v,\nu_u)(0)\\
=&
\dfrac{1}{2\ep'}\det(f_{uuu},f_{uu},\nu)(0)
-\dfrac{1}{\ep'}\det(f_{uu},f_{uuu},\nu)(0)\\
=&
-\dfrac{3}{2\ep'}\det(f_{uu},f_{uuu},\nu)(0).
\end{align*}
One can easily show that $0$ is a singular point of the
second kind if and only if $d\lambda\ne0$, $\eta\lambda=0$
and $\eta\eta\lambda\ne0$ at $(0,0)$. 
Thus the assertion follows.
\end{proof}

Let $0$ be a singular point of the second kind 
of a frontal
$f:(\R^2,0)\to(\R^3,0)$.
By Lemma \ref{lem:23type},
$\det(\hat\bgamma'',\hat\bgamma''',\nu)\ne0$ holds,
we take a parameter of $\gamma$ satisfying
$\det(\hat\bgamma'',\hat\bgamma''',\nu)>0$.
Again by Lemma \ref{lem:23type},
$\hat\bgamma''(0)\ne0$,
the tangent line of $\hat{\bgamma}(u)=f\circ \gamma(u)$ at $0$ is
well-defined.
Set a unit vector field $\e(u)$ along $\gamma$ 
such that $\e(u)$ is tangent to $\hat\bgamma$ if $u\ne0$
which
satisfies
$$\e(0)=\lim_{u\to+0}\dfrac{\hat\bgamma'(u)}{|\hat\bgamma'(u)|}.$$
We set
$$\bnu(u)=\nu\circ\gamma(u)\quad\text{and}\quad
\bb(u)=-\e(u)\times \bnu(u).$$
Then
$
\{\e,\bb,\bnu\}
$
forms a positive orthonormal frame along $\gamma$.
We have the following Frenet-Serret type formula:
\begin{equation}\label{eq:frenet}
\left\{
\begin{array}{ccl}
\displaystyle
\e'(u) &=& \tilde \kappa_g(u) \bb(u)+\tilde \kappa_\nu(u) \bnu(u), 
\\
\displaystyle
\bb'(u) &=& -\tilde \kappa_g(u) \e(u)+\tilde \kappa_t(u) \bnu(u), 
\\
\displaystyle
\bnu'(u) &=& -\tilde \kappa_\nu(u) \e(u)-\tilde \kappa_t(u) \bb(u).
\end{array}
\right. 
\end{equation}
%Here, we denote ${~}'=\dfrac{d}{du}$. 
Note that the above invariants depend on the choice of parameter.
These invariants can be written by
using known invariants.
\begin{proposition}\label{prop:frenet}
The invariants\/
$\tilde \kappa_g$,
$\tilde \kappa_\nu$ and\/
$\tilde \kappa_t$ satisfy
\begin{equation}\label{eq:frenetinv}
\tilde \kappa_g=|\alpha|\kappa_s,\quad
\tilde \kappa_\nu=\alpha\kappa_\nu,\quad
\tilde \kappa_t=\alpha\kappa_t,
\end{equation}
where 
$$
\alpha(u)=\sgn(u)|\gamma'(u)|,
$$
and which is a\/ $C^\infty$ function.
Here\/ 
$\kappa_s$ is the {\it singular curvature\/} {\rm (\cite{SUY})},
$\kappa_\nu$ is the {\it limiting normal curvature\/} {\rm (\cite{SUY})} 
and
$\kappa_t$ is the {\it cuspidal torsion\/} {\rm (\cite{MS})}.
\end{proposition}
We define
$
\kappa_g(u)=\sgn(\alpha(u))\kappa_s(u)
$,
and call the {\it geodesic curvature}.
Then 
$\tilde \kappa_g=\alpha\kappa_g$.
We take 
the coordinate system $(u,v)$ satisfying
$S(f)=\{v=0\}$. 
Setting 
$\tilde u=u$, $\tilde v=\int_0^v |f_v(u,v)|\,dv$,
one can see the coordinate system $(\tilde u,\tilde v)$ 
satisfies
$S(f)=\{\tilde v=0\}$ and $|f_{\tilde v}(\tilde u,\tilde v)|=1$.
Let $(u,v)$ be a coordinate system satisfying
$S(f)=\{v=0\}$ and $|f_v(u,v)|=1$.
We take the null vector field
$$
\eta(u)=\partial_u+\ep(u)\partial_v\quad(
\ep(0)=0,\ \ep'(0)>0)
$$
as above. Then $\ep=\alpha$.
Since $0$ is non-degenerate,
$\lambda_v\ne0$.
Thus $\det(g,f_v,\nu)\ne0$ at $0$.
\begin{lemma}\label{lem:enu}
Under the above settings, we have
$$
f_v(u,0)=-\e(u),\quad
f_v\times  g/|f_v\times  g|=\bnu(u).
$$
\end{lemma}
\begin{proof}
We remark that by the assumption \eqref{eq:ep'}, $\sgn(u)\ep(u)>0$. 
By \eqref{eq:dfeta}, 
$$
\be(0)
=\lim_{u\to+0}\dfrac{f_u(u,0)}{|f_u(u,0)|}
=\lim_{u\to+0}\dfrac{-\ep(u)}{|\ep(u)|}f_v(u,0)
=-f_v(0,0),
$$
and hence it holds that $\be(u)=-f_v(u,0)$.
Next, since $\eta f=vg$,
\begin{equation}\label{eq:bnu}
\bnu(u)=
\pm \dfrac{f_v\times g}{|f_v\times g|}(u,0).
\end{equation}
On the other hand, by \eqref{eq:dfetabibunu},
$$
\det(f_{uu},f_{uuu},f_v\times g)
=
\det(-\ep' f_v,
-(\ep'' f_v+2\ep' f_{uv}),f_v\times g)
=
2(\ep')^2\det(f_v,
g,f_v\times g),
$$
we see that the $\pm$ sign in \eqref{eq:bnu}
should be $+$.
\end{proof}
\begin{proof}[Proof of Proposition {\rm \ref{prop:frenet}}]
Let $(u,v)$ be a coordinate system just after
Proposition \ref{prop:frenet}.
We see
$$
\tilde \kappa_g=\inner{\e'}{\bb}=-\inner{f_{uv}}{\bb}
=\det(f_{uv},\e,\nu)
=-\det(f_v,f_{uv},\nu).
$$
By the definition of the singular curvature
(\cite[(1.7)]{SUY}), and by \eqref{eq:dfetabibunu},
we have
$$
\kappa_s=\sgn(\ep)\sgn(\eta\lambda)\dfrac{\det(f_u,f_{uu},\bnu)}{|f_u|^3}
=\sgn(\ep)\sgn(\eta\lambda)\dfrac{\det(f_v,f_{uv},\bnu)}{|\ep|}. 
$$
Moreover,
$$
\eta\lambda=\eta\det(f_u,f_v,\bnu)
=\det(\eta f_u,f_v,\bnu)+\det(f_u,\eta f_v,\bnu)
=\det(-\ep f_v,g,\bnu)
=-\ep$$
holds on the $u$-axis.
Thus
$
\kappa_s=-\det(f_v,f_{uv},\bnu)/|\ep|,
$
and we have $\tilde\kappa_g=|\ep|\kappa_s$.

Next we consider $\tilde\kappa_\nu$.
By Lemma \ref{lem:enu} and \eqref{eq:dfetabibunu},
$$\tilde\kappa_\nu
=\inner{\e'}{\bnu}
=-\inner{f_{uv}}{\bnu}
=-\inner{g-\ep f_{vv}}{\bnu}
=\ep\inner{f_{vv}}{\bnu}.
$$
On the other hand, by the definition 
of the limiting normal curvature (\cite[(3.11)]{SUY})
and \eqref{eq:dfetabibunu},
$$
\kappa_\nu=
\dfrac{\inner{f_{uu}}{\bnu}}{|f_u|^2}
=
\dfrac{\inner{\ep^2 f_{vv}}{\bnu}}{|\ep|^2}
=
\inner{f_{vv}}{\bnu}
$$
holds, thus we have
$
\tilde\kappa_\nu=\ep\kappa_\nu
$. Finally, by Lemma \ref{lem:enu} and \eqref{eq:dfetabibunu},
\begin{equation}\label{eq:kappat100}
\begin{array}{rcl}
\tilde\kappa_t
&=&\inner{\bb'}{\bnu}
=\det(\be,\bnu,\bnu')\\
&=&
\dfrac{\det(f_v,f_v\times g,(f_v\times g)_u)}{|f_v\times g|^2}\\[3mm]
&=&
\dfrac{\det(f_v,f_v\times g,f_{uv}\times g)}{|f_v\times g|^2}
+
\dfrac{\det(f_v,f_v\times g,f_v\times g_u)}{|f_v\times g|^2}\\[3mm]
&=&
-\ep\dfrac{\det(f_v,f_v\times g,f_{vv}\times g)}{|f_v\times g|^2}
+
\dfrac{\det(f_v,f_v\times g,f_v\times g_u)}{|f_v\times g|^2}\\[3mm]
&=&
-
\ep\dfrac{\det(g,f_{vv},f_v)\inner{f_v}{g}}{|f_v\times g|^2}
-
\dfrac{\det(g,g_u,f_v)}{|f_v\times g|^2},
\end{array}
\end{equation}
where we used the formula
$
\det(a\times b,a\times c,d)
=
\det(a,b,c)\inner{a}{d}
$ $(a,b,c,d\in\R^3)$.
On the other hand, by the definition
of the cuspidal torsion (\cite[(5.1)]{MS}), \eqref{eq:dfeta} and 
\eqref{eq:dfetabibunu},
$$
\begin{array}{rcl}
\kappa_t
&=&
\dfrac{\det(f_u,\eta\eta f,\eta\eta f_u)}{|f_u\times \eta\eta f|^2}
-
\dfrac{\det(f_u,\eta\eta f,f_{uu})\inner{f_u}{\eta\eta f}}
{|f_u|^2|f_u\times \eta\eta f|^2}\\[3mm]
&=&
\dfrac{\det(-\ep f_v,\ep g, \ep g_u)}{\ep^4|f_v\times g|^2}
-
\dfrac{\det(-\ep f_v,\ep g,\ep^2 f_{vv})\inner{-\ep f_v}{\ep g}}
{\ep^6|f_v\times g|^2}\\[3mm]
&=&
-
\dfrac{\det(f_v,g,g_u)}{\ep|f_v\times g|^2}
-
\dfrac{\det(f_v,g,f_{vv})\inner{f_v}{g}}
{|f_v\times g|^2}
\end{array}
$$
holds, thus we have $\tilde\kappa_t=\ep \kappa_t$.
\end{proof}
\begin{remark}
Our formula \eqref{eq:frenet} 
depends on the choice of the parameter.
Usually the Frenet-Serret formula is written
by the arclength parameter.
However, in our case the arclength parameter
is not differentiable.
In \cite{su}, the half arc-length parameter
for plane cusp
was introduced. 
It is also well-defined and differentiable in our case.
If we take the half arc-length parameter $h$,
then the function $\alpha$ in \eqref{eq:frenetinv} is 
equal to $h$.
\end{remark}
\begin{remark}
We remark that
invariants of frames along curves in $\R^3$ with
singularities (framed singular curves) are studied in \cite{h}.
The frame considered in the present paper is
constructed by using the normal vector of given frontal.
So our invariants are related to the geometry of
the frontal.
See \cite{h} for the study of the 
properties of 
invariants of framed singular curves itself.
\end{remark}

\section{Developable surfaces along singular set}
\label{sec:devsurfs}
Let $f:(\R^2,0) \to (\R^{3},0)$ be a frontal and $\nu$ its
unit normal, and let $0$ be a singular point
of the second kind.
Throughout in this section, 
we take the coordinate system $(u,v)$ near $0$
satisfying $S(f)=\{v=0\}$.
Let $\{\e,\bb,\bnu\}$ be the Darboux frame 
defined in Subsection \ref{sec:frame}.
In this section, following \cite{HII,IzuOta,ist},
we consider developable surfaces along $S(f)$.
Developable surfaces along curves with singularities are considered in
\cite{h}.
See \cite{gray,Pot} for basic notions for ruled surfaces,
and \cite{Iz-Takeruled,Iz-Take,izu-take} for singularities of
ruled surfaces.
\subsection{Osculating developable surfaces}\label{sec:odf}
%%%%%%%%%%%%%%%%%%%
We assume that
$(\tilde\kappa_{\nu}(u),\tilde\kappa_t (u))\not= (0,0)$
in a small neighborhood of $0$.
By \eqref{prop:frenet}, $\tilde\kappa_\nu(0)=0$,
this assumption is equivalent to $\tilde\kappa_t (0)\ne0$.
Under this assumption,
we define a ruled surface
$
OD_{f}:I\times\R\lon\R^3
$
by
\begin{equation*}
OD_{f}(u,t)=f(u,0)+t\overline{D_o}(u)\quad
\left(\overline{D_o}(u)=
\frac{\tilde\kappa_t(u)\bm{e} (u)-\tilde\kappa_{\nu}(u)\bb(u)}
{\sqrt{\tilde\kappa_t(u)^2+\tilde\kappa_{\nu}(u)^2}}
\right),
\end{equation*}
and call an {\it osculating developable surface along\/} $f$.
Set
\begin{align}\label{eq:deltao}
\tilde\delta_o&=
 \tilde\kappa_g\big((\tilde\kappa_\nu)^2+(\tilde\kappa_t)^2\big)
 -\tilde\kappa_t(\tilde\kappa_\nu)'+(\tilde\kappa_t)'\tilde\kappa_\nu,
\end{align}
where ${~}'=d/d u$.
By \eqref{eq:frenet}, we see
\begin{equation}\label{eq:dop}
\overline{D_o}' 
=
\dfrac{\tilde\delta_o}{(\tilde\kappa_t^2+\tilde\kappa_{\nu}^2)^{3/2}}
(\tilde\kappa_\nu\be+\tilde\kappa_t\bb)
\end{equation}
and
$\det \big(\hat\bgamma', \overline{D_o}, \overline{D_o}'\big)=0$,
it holds that
$OD_{f}$ is a developable surface.
Setting 
$\lambda_o=\tilde\delta_o t + \tilde\kappa_{\nu}\ep
(\tilde\kappa_{\nu}^2 + \tilde\kappa_t^2)^{1/2}$,
it holds that
$
S(OD_f)=\{
\lambda_o(u,t)=0\}
$.
If $\tilde\delta_o(0)=0$, then
all the points on the ruling passing through $\hat\bgamma(0)$ are singular value.
When $\tilde\delta_o\ne0$, we set
\begin{equation}\label{eq:strsourod}
t_o(u)=-
\dfrac{
\tilde\kappa_{\nu} \ep
(\tilde\kappa_{\nu}^2 + \tilde\kappa_t^2)^{1/2}}{\tilde\delta_o},
\end{equation}
and
$\bs_o(u)=OD_f(u,t_o(u))$.
Then
\begin{equation}\label{eq:stro}
\bs_o
=
\hat\bgamma-
\dfrac{\inner{\hat\bgamma'}{\overline{D_o}'}}
{\inner{\overline{D_o}'}{\overline{D_o}'}}\overline{D_o}
\end{equation}
holds, and thus $\bs_o(u)$ is the striction curve 
(cf. \cite[Section 17.3]{gray}) of
$OD_f$.
In this case, we have
\begin{equation}\label{eq:sodp}
t_o'
=
-\dfrac{\tilde\sigma_o (\tilde\kappa_\nu^2+\tilde\kappa_t^2) 
+ \ep \tilde\kappa_t \tilde\delta_o^2}
{\sqrt{\tilde\kappa_\nu^2+\tilde\kappa_t^2}\tilde\delta_o^2},
\quad
\bs_o'
=
-\dfrac{\tilde\sigma_o}{\tilde\delta_o^2}
(\tilde\kappa_t \be - \tilde\kappa_\nu \bb),
\end{equation}
where
\begin{align*}
\tilde\sigma_o&=
-\ep \tilde\kappa_\nu \tilde\delta_o' 
- 
(-\ep' \tilde\kappa_\nu + \ep \tilde\kappa_g \tilde\kappa_t - 
2 \ep \tilde\kappa_\nu') \tilde\delta_o
\\
&=
\tilde\kappa_\nu \ep' 
\Big(\tilde\kappa_g(\tilde\kappa_\nu^2+\tilde\kappa_t^2) 
- \tilde\kappa_t \tilde\kappa_\nu'
+ \tilde\kappa_\nu \tilde\kappa_t'\Big) 
- 
\ep 
\Big(
 \tilde\kappa_t (2(\tilde\kappa_\nu')^2 
-\tilde\kappa_\nu \tilde\kappa_\nu'' )
+\tilde\kappa_\nu (-2 \tilde\kappa_\nu' \tilde\kappa_t' 
+\tilde\kappa_\nu \tilde\kappa_t'') \\
&\hspace{15mm}
+ \tilde\kappa_\nu(\tilde\kappa_\nu^2 + \tilde\kappa_t^2) \tilde\kappa_s'
+
3 \tilde\kappa_t (-\tilde\kappa_t \tilde\kappa_\nu' + 
    \tilde\kappa_\nu \tilde\kappa_t') \tilde\kappa_g + 
 \tilde\kappa_t (\tilde\kappa_\nu^2 + \tilde\kappa_t^2) \tilde\kappa_g^2
\Big).
\end{align*}
We have the following characterization of 
singularities of the osculating developable surface
using $\tilde\delta_o$ and $\tilde\sigma_o$.
Except for $u=0$, singular points are cuspidal edges,
we stick to our consideration to $u=0$.
We have $t_o(0)=0$,
\begin{equation}\label{eq:deltaop0}
\tilde\delta_o=
\tilde\kappa_t\big(
\tilde\kappa_g\tilde\kappa_t-\tilde\kappa_\nu'
\big),\quad
\tilde\delta_o'=
\tilde\kappa_t\big(
\tilde\kappa_t\tilde\kappa_g'
+
2 \tilde\kappa_g \tilde\kappa_t'
-
\tilde\kappa_\nu''\big)
\quad
\text{at}\quad u=0
\end{equation}
and
\begin{equation}\label{eq:sigmaop0}
\tilde\sigma_o
=
0,\quad
\tilde\sigma_o'
=
\tilde\kappa_t\ep'
\big(
\tilde\kappa_g\tilde\kappa_t-3 \tilde\kappa_\nu'
\big)
\big(
\tilde\kappa_g\tilde\kappa_t- \tilde\kappa_\nu'
\big)
\quad
\text{at}\quad u=0.
\end{equation}
\begin{theorem}\label{thm:odsing} 
We assume that
$\tilde\kappa_t(0)\ne0$.
If\/ $OD_f$ satisfies\/ $\tilde\delta_o(0)\ne0$,
$($namely, $\tilde\kappa_g(0) \tilde\kappa_t(0) -\tilde\kappa_\nu'(0)\ne0)$,
then the singular point\/ $(0,0)$ 
of\/ $OD_{f}$ is
\begin{enumerate}
\item\label{itm:ce} 
never be a cuspidal edge,
\item\label{itm:sw} swallowtail if and only if\/
$\tilde\kappa_g(0) \tilde\kappa_t(0) -3\tilde\kappa_\nu'(0)\ne0$.
\end{enumerate}
If\/ $OD_f$ satisfies\/ $\tilde\delta_o(0)=0$ 
$($namely, $\tilde\kappa_g(0) \tilde\kappa_t(0) -\tilde\kappa_\nu'(0)=0)$,
then the singular point\/ $(0,0)$ 
of\/ $OD_{f}$ is
\begin{enumerate}
\setcounter{enumi}{2}
\item\label{itm:cbkodf} cuspidal beaks if and only if\/
$\big(\tilde\kappa_t\tilde\kappa_g'+2 \tilde\kappa_g \tilde\kappa_t'
-\tilde\kappa_\nu''\big)(0)\ne0$, and\/ $\tilde\kappa_\nu'(0)\ne0$.
\end{enumerate}
\end{theorem}
\begin{proof}
By \eqref{eq:strsourod},
and by a calculation,
the null vector field $\eta_o$ of $OD_f$ is
$$
\eta_o
=
(\tilde\kappa_\nu^2+\tilde\kappa_t^2)
\partial_u
-\ep \tilde\kappa_t (\tilde\kappa_\nu^2+\tilde\kappa_t^2)^{1/2}
\partial_t.
$$
Since $\bnu$ is a unit normal vector to $OD_f$ and
$\eta_o\bnu(0)\ne0$, $OD_f$ is front at $(0,0)$.
The function $\phi$ in \eqref{eq:criphi} is
\begin{equation}\label{eq:phiod}
\phi(u)=\det\pmt{
1&t_o'\\
\tilde\kappa_\nu^2+\tilde\kappa_t^2&
-\ep \tilde\kappa_t (\tilde\kappa_\nu^2+\tilde\kappa_t^2)^{1/2}
}=
\dfrac{\tilde\sigma_o(\tilde\kappa_\nu^2+\tilde\kappa_t^2)^{3/2}}
{\tilde\delta_o^2}.
\end{equation}
Since
$\ep(0)=0$, 
the condition $\phi(0)=0$ is equivalent to $\tilde\sigma_o(0)=0$,
and the condition $\phi(0)=\phi'(0)=0$ is equivalent to
$\tilde\sigma_o(0)=\tilde\sigma_o'(0)=0$.
Since always $\tilde\sigma_o(0)=0$ holds,
we have the assertion \eqref{itm:ce}.
By \eqref{eq:sigmaop0}, when
$\tilde\delta_o(0)\ne0$, then
$\tilde\sigma_o'(0)\ne0$ is equivalent to
$\tilde\kappa_g(0) \tilde\kappa_t(0) -3\tilde\kappa_\nu'(0)\ne0$.
Thus we have the assertion \eqref{itm:sw}.
If $\tilde\delta_0(0)=0$, then
$d\lambda_o=0$ holds, and
since $(\lambda_o)_{tt}=0$,
$\det\hess \lambda_o (0,0)<0$ is equivalent to
$(\lambda_o)_{ut}\ne0$, and it is equivalent to 
$\tilde\delta_o'(0)\ne0$.
Moreover, $\eta_o\eta_o\lambda_o(0,0)\ne0$ is
equivalent to $\tilde\kappa_\nu'(0)\ne0$.
Thus we have the assertion \eqref{itm:cbkodf}.
\end{proof}
\subsection{Normal developable surfaces}\label{sec:ndf}
%%%%%%%%%%%%%%%%%%%%%%%%%%%%%
We assume that
$(\tilde\kappa_t(u),\tilde\kappa_g(u))\not= (0,0)$.
Under this assumption,
we define a ruled surface
$
ND_{f}:I\times \R\lon \R^3
$
by
\[
ND_{f}(u,t)=f (u,0)+t\overline{D_n}(u)\quad
\left(\overline{D_n}(u)=
\frac{\tilde\kappa_t(u)\bm{e} (u)+\tilde\kappa_g(u)\bm{\nu}(u)}
{\sqrt{\tilde\kappa_t(u)^2+\tilde\kappa_g(u)^2}}\right).
\]
and call a {\it normal developable surface along\/} $f$.
Set
\begin{equation}\label{eq:deltab}
\tilde\delta_n=\tilde\kappa_\nu(\tilde\kappa_g^2 + \tilde\kappa_t^2) 
-\tilde\kappa_g\tilde\kappa_t' + 
 \tilde\kappa_t\tilde\kappa_g'.
\end{equation}
By \eqref{eq:frenet}, we see
\begin{equation}\label{eq:dnp}
\overline{D_n}
' 
=
\dfrac{\tilde\delta_n}{(\tilde\kappa_t ^2 + \tilde\kappa_g ^2)^{3/2}}
(-\tilde\kappa_g \bm{e} + \tilde\kappa_t \bm{\nu}),
\end{equation}
and 
$\det \big(\hat\bgamma', \overline{D_n}, \overline{D_n}'\big)=0$,
it holds that
$ND_{f}$ is a developable surface.
Setting 
$\lambda_n=-\tilde\delta_n t + \tilde\kappa_g\ep
(\tilde\kappa_g^2 + \tilde\kappa_t^2)^{1/2}$,
it holds that
$
S(ND_f)=\{
\lambda_n(u,t)=0\}
$.
If $\tilde\delta_n(0)=0$, then
all the points on the ruling passing through $\hat\bgamma(0)$ are singular value.
When $\tilde\delta_n\ne0$, we set
$$
t_n(u)=
\dfrac{
\tilde\kappa_{g} \ep
(\tilde\kappa_{g}^2 + \tilde\kappa_t^2)^{1/2}}{\tilde\delta_n}
$$
and
$\bs_n=ND_f(u,t_n(u))$.
Then
\begin{equation}\label{eq:strn}
\bs_n
=
\hat\bgamma-
\dfrac{\inner{\hat\bgamma'}{\overline{D_n}'}}
{\inner{\overline{D_n}'}{\overline{D_n}'}}\overline{D_n}
\end{equation}
holds, and thus $\bs_n(u)$ is the striction curve of
$ND_f$.
In this case, we have
\begin{equation}\label{eq:sndp}
t_n'
=
\dfrac{\sigma_n (\tilde\kappa_g^2+\tilde\kappa_t^2) 
- \ep \tilde\kappa_t \delta_n^2}
{\sqrt{\tilde\kappa_g^2+\tilde\kappa_t^2}\delta_n^2},
\quad
\bs_n'
=
\dfrac{\tilde\sigma_n}{\tilde\delta_n^2}
(\tilde\kappa_t \be + \tilde\kappa_g \bnu),
\end{equation}
where
\begin{align*}
\tilde\sigma_n&=
-\ep \tilde\kappa_g \tilde\delta_n' 
+ 
(\ep' \tilde\kappa_g + \ep \tilde\kappa_g \tilde\kappa_t + 
2 \ep \tilde\kappa_g') \tilde\delta_n
\\
&=
\tilde\kappa_g \ep' 
\Big(\tilde\delta_n (\tilde\kappa_g^2 + \tilde\kappa_t^2) 
+ 
\tilde\kappa_t \tilde\kappa_g' - 
\tilde\kappa_g \tilde\kappa_t'\Big) 
+
\ep
\Big(
 \tilde\kappa_t (2 \tilde\kappa_g'^2 
-\tilde\kappa_g \tilde\kappa_g'') 
+\tilde\kappa_g (-2 \tilde\kappa_g' \tilde\kappa_t' + 
      \tilde\kappa_g \tilde\kappa_t'') \\
&\hspace{15mm}
- 
 \tilde\kappa_g (\tilde\kappa_g^2 + \tilde\kappa_t^2) \tilde\kappa_\nu'
+
3 \tilde\kappa_t (\tilde\kappa_t \tilde\kappa_g' - 
    \tilde\kappa_g \tilde\kappa_t') \tilde\kappa_\nu + 
 \tilde\kappa_t (\tilde\kappa_g^2 + \tilde\kappa_t^2) \tilde\kappa_\nu^2
\Big).
\end{align*}
Similarly to the case of the osculating developable surface
we have the following characterization of 
singularities of the normal developable surfaces.
We also note that $t_n(0)=0$,
\begin{equation}\label{eq:deltanp0}
\tilde\delta_n=
\tilde\kappa_g'\tilde\kappa_t-\tilde\kappa_g\tilde\kappa_t'
,\quad
\tilde\delta_n'=
\tilde\kappa_\nu'(\tilde\kappa_g^2+\tilde\kappa_t^2)
+
\tilde\kappa_g''\tilde\kappa_t-\tilde\kappa_g\tilde\kappa_t''
\quad
\text{at}\quad u=0
\end{equation}
and
\begin{equation}\label{eq:sigmanp0}
\tilde\sigma_n
=
\ep'\tilde\kappa_g
(\tilde\kappa_g'\tilde\kappa_t-\tilde\kappa_g\tilde\kappa_t')
,\quad
\tilde\sigma_n'
=
(\tilde\kappa_g'\tilde\kappa_t-\tilde\kappa_g\tilde\kappa_t')
(3\ep'\tilde\kappa_g'+\ep''\tilde\kappa_g)
\quad
\text{at}\quad u=0.
\end{equation}
\begin{theorem}\label{thm:ndsing} 
We assume that
$(\tilde\kappa_t(0),\tilde\kappa_g(0))\not= (0,0)$.
If\/ $ND_f$ satisfies\/ $\tilde\delta_n(0)\ne0$,
$($namely, 
$\tilde\kappa_g'(0)\tilde\kappa_t(0)-
\tilde\kappa_g(0)\tilde\kappa_t'(0)
\ne0)$,
then the singular point\/ $(0,0)$ 
of\/ $ND_{f}$ is 
\begin{enumerate}
\item\label{itm:cen} 
cuspidal edge if and only if\/
$\tilde\kappa_g\ne0$ holds,
\item\label{itm:swn} swallowtail if and only if\/
$\tilde\kappa_g=0$ holds.
\end{enumerate}
If\/ $ND_f$ satisfies\/ $\tilde\delta_n(0)=0$
$($namely, 
$\tilde\kappa_g'(0)\tilde\kappa_t(0)-
\tilde\kappa_g(0)\tilde\kappa_t'(0)
=0)$,
then the singular point\/ $(0,0)$ 
of\/ $ND_{f}$ is not a cuspidal beaks.
\end{theorem}
\begin{proof}
By \eqref{eq:strsourod},
and a calculation,
the null vector field $\eta_n$ of $ND_f$ is
$$
\eta_n
=
(\tilde\kappa_g^2+\tilde\kappa_t^2)
\partial_u
-\ep \tilde\kappa_t (\tilde\kappa_g^2+\tilde\kappa_t^2)^{1/2}
\partial_t.
$$
Since $\bb$ is a unit normal vector to $ND_f$ and
$\eta_n\bb(0)\ne0$, $ND_f$ is front at $(0,0)$.
The function $\phi$ in \eqref{eq:criphi} is
\begin{equation}\label{eq:phind}
\phi(u)=\det\pmt{
1&t_n'\\
\tilde\kappa_g^2+\tilde\kappa_t^2&
-\ep \tilde\kappa_t (\tilde\kappa_g^2+\tilde\kappa_t^2)^{1/2}
}=
-\dfrac{\tilde\sigma_n(\tilde\kappa_g^2+\tilde\kappa_t^2)^{3/2}}
{\tilde\delta_n^2}.
\end{equation}
Since
$\ep(0)=0$, 
the condition $\phi(0)=0$ is equivalent to 
$\tilde\sigma_n(0)=
\big(\ep'\tilde\kappa_g 
(\tilde\kappa_t\tilde\kappa_g'-\tilde\kappa_g\tilde\kappa_t')\big)(0)=0$,
and the condition $\phi(0)=\phi'(0)=0$ is equivalent to
$\tilde\sigma_n(0)=0$, and
$\tilde\sigma_n'=(\tilde\kappa_t\tilde\kappa_g'-\tilde\kappa_g\tilde\kappa_t')
(3\ep'\tilde\kappa_g'+\tilde\kappa_g \ep'')(0)=0$.
Since 
$\tilde\delta_n
=\tilde\kappa_t\tilde\kappa_g'-\tilde\kappa_g\tilde\kappa_t'$ at $0$,
$\tilde\kappa_g'(0)\ne0$ holds under $\tilde\kappa_g(0)=0$.
Thus
the assertions \eqref{itm:cen} and \eqref{itm:swn} hold.
Since 
$(\lambda_n)_{u}=\ep'\tilde\kappa_g(\tilde\kappa_g^2+\tilde\kappa_t^2)^{1/2}$,
$(\lambda_n)_{t}=\tilde\kappa_t\tilde\kappa_g'-\tilde\kappa_g\tilde\kappa_t'$
holds, $d\lambda_n(0,0)=0$ under the assumption 
$(\tilde\kappa_t(0),\tilde\kappa_g(0))\not= (0,0)$
is equivalent to $\tilde\kappa_g(0)=\tilde\kappa_g'(0)=0$.
A necessary condition that $0$ is a cuspidal beaks
is 
$\lambda_n''(0,0)\ne0$.
However it does not hold under the condition
$\tilde\kappa_g(0)=\tilde\kappa_g'(0)=0$.
Thus we have the last assertion.
\end{proof}

Here we give two examples.
\begin{example}[Standard swallowtail]
The standard swallowtail is $f:(u,v)\mapsto(u,4v^3+2uv,3v^4+uv^2)$.
%$$
%f(u,v)=
%(-6 u^2+v,4 u^3+2 u (-6 u^2+v),3 u^4+u^2 (-6 u^2+v)).
%$$
The normal developable surface $ND_f$ can be
written as Figure \ref{fig:odfsw11}.
Since $f$ is a tangent developable surface, 
$OD_f$ coincides with $f$.
\begin{figure}[!ht]
\centering
  \includegraphics[width=.2\linewidth]{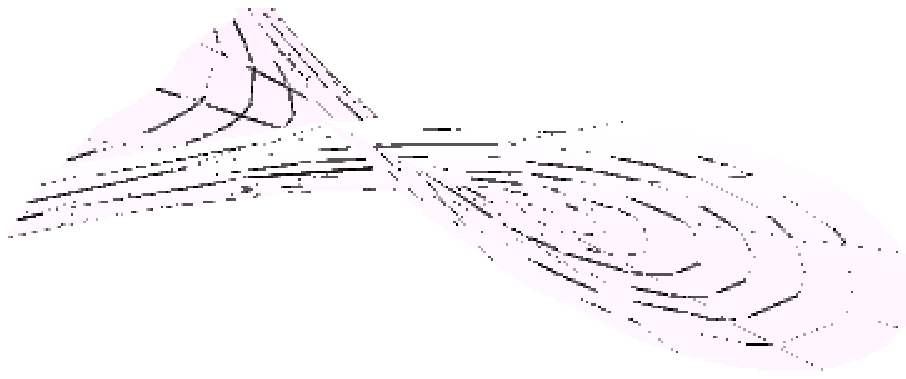}
% \hspace{2mm}
%  \includegraphics[width=.2\linewidth]{fig01-02odfs.eps}
 \hspace{2mm}
  \includegraphics[width=.2\linewidth]{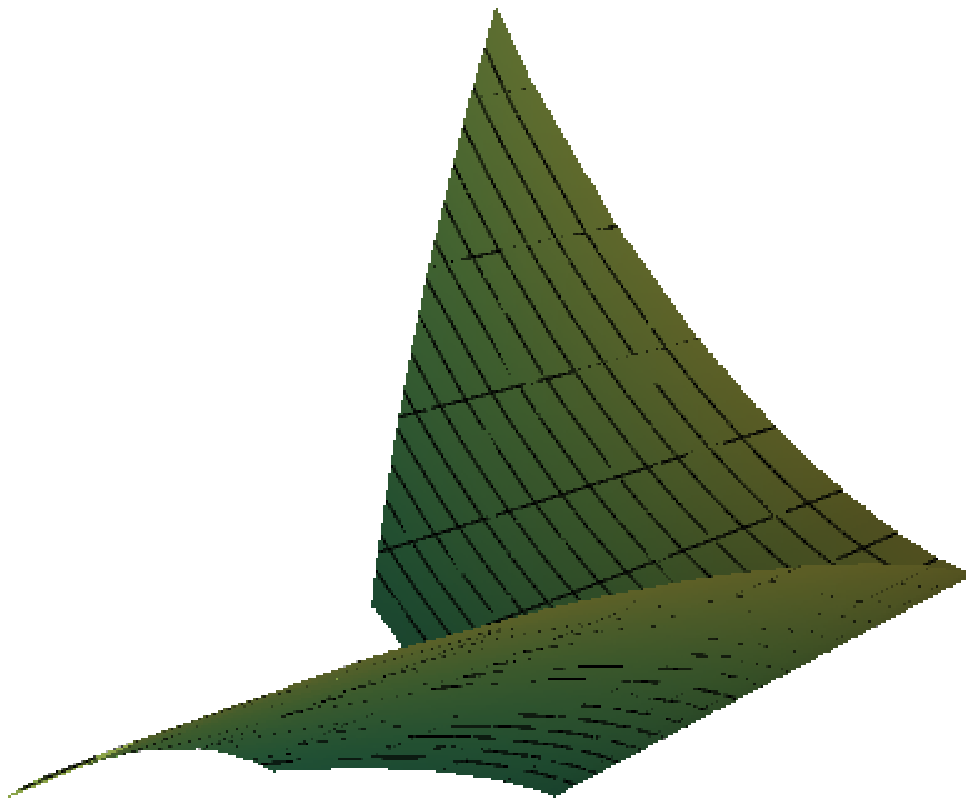}
 \hspace{2mm}
  \includegraphics[width=.2\linewidth]{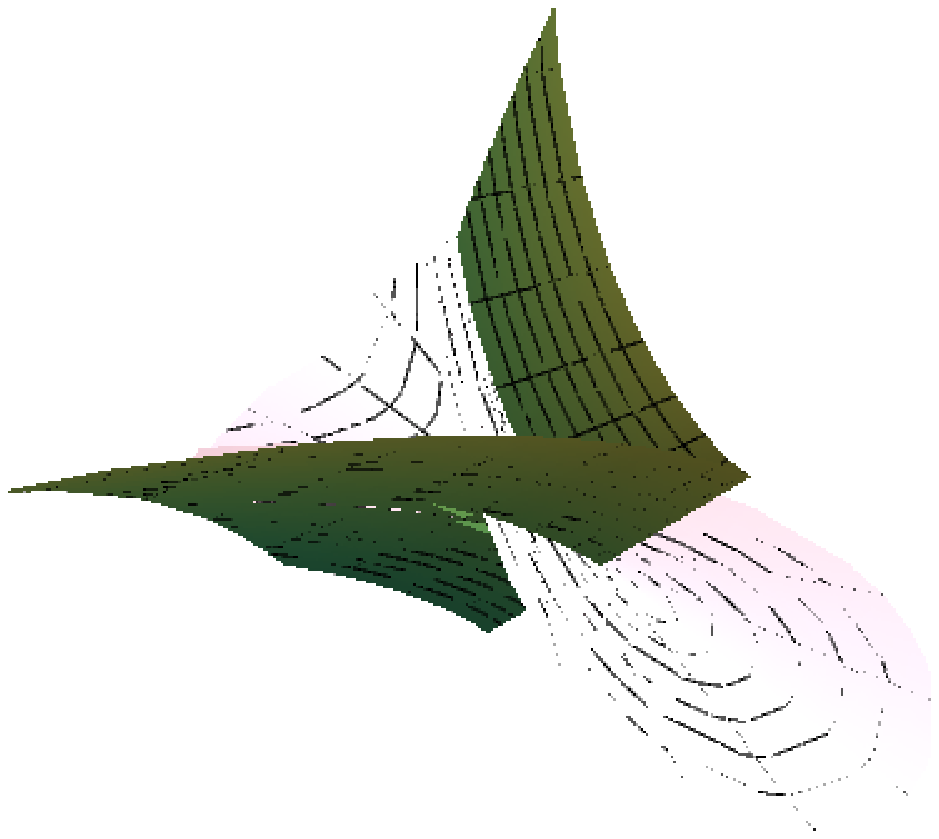}
\caption{Standard swallowtail (white), normal developable (green) 
and both of them}
\label{fig:odfsw11}
\end{figure}
\end{example}
\begin{example}
Let us set
\begin{equation}\label{eq:swex}
f(u, v) = 
\left(v + \dfrac{u^2}{2} - \dfrac{u^2 v}{2} - \dfrac{u^4}{8}, 
\dfrac{u^3}{3} + u v,  \dfrac{v^2}{2}\right).
\end{equation}
The osculating and normal developable surfaces
of $f$ are drawn in Figures
\ref{fig:odfsw21} and \ref{fig:odfsw22}.
\begin{figure}[!ht]
\centering
  \includegraphics[width=.2\linewidth]{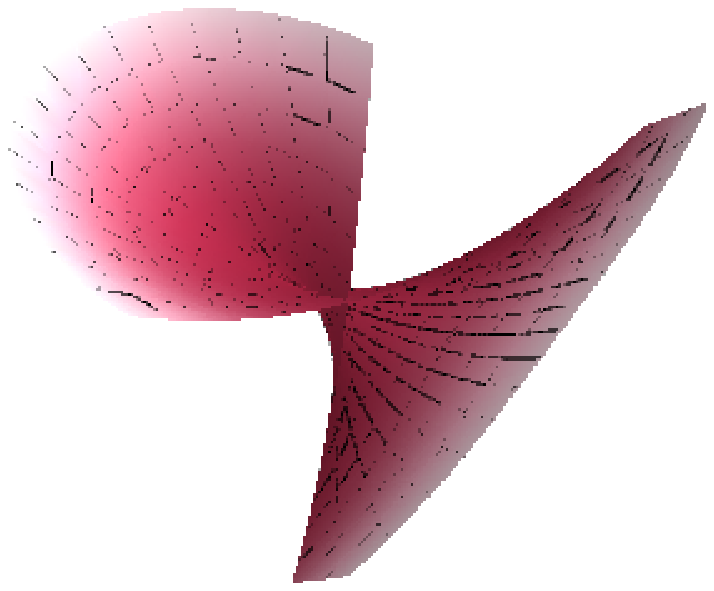}
 \hspace{2mm}
  \includegraphics[width=.2\linewidth]{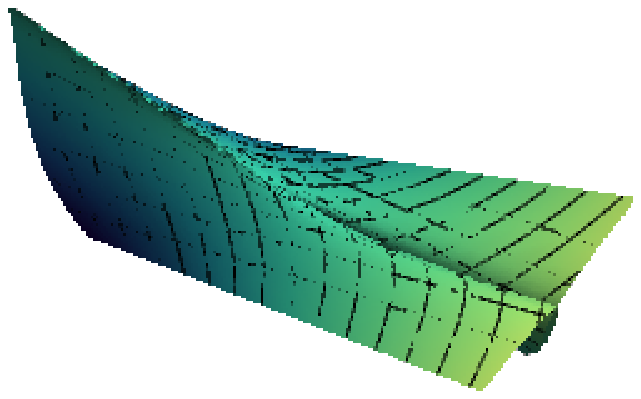}
 \hspace{2mm}
  \includegraphics[width=.2\linewidth]{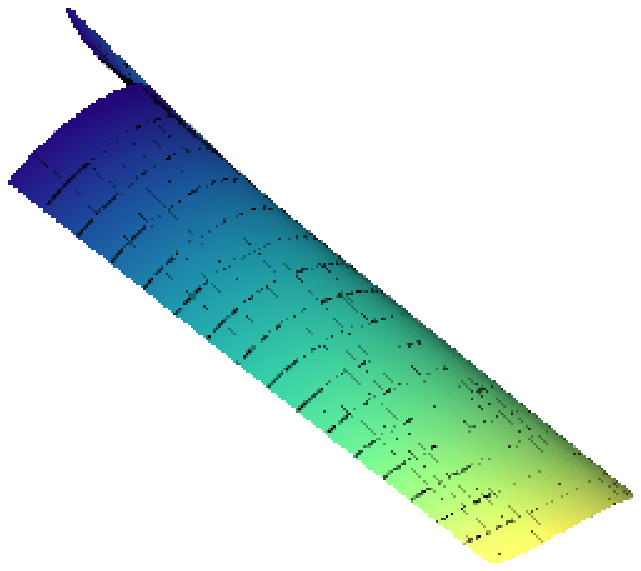}
\caption{Swallowtail given in \eqref{eq:swex} and its
osculating and normal developables}
\label{fig:odfsw21}
\end{figure}
\begin{figure}[!ht]
\centering
  \includegraphics[width=.2\linewidth]{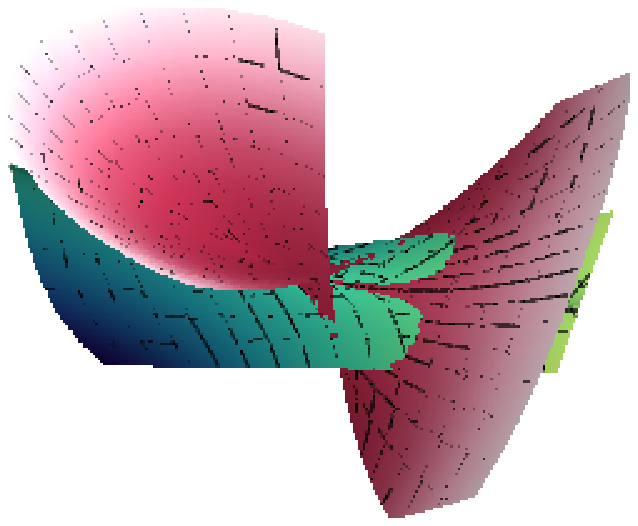}
 \hspace{2mm}
  \includegraphics[width=.2\linewidth]{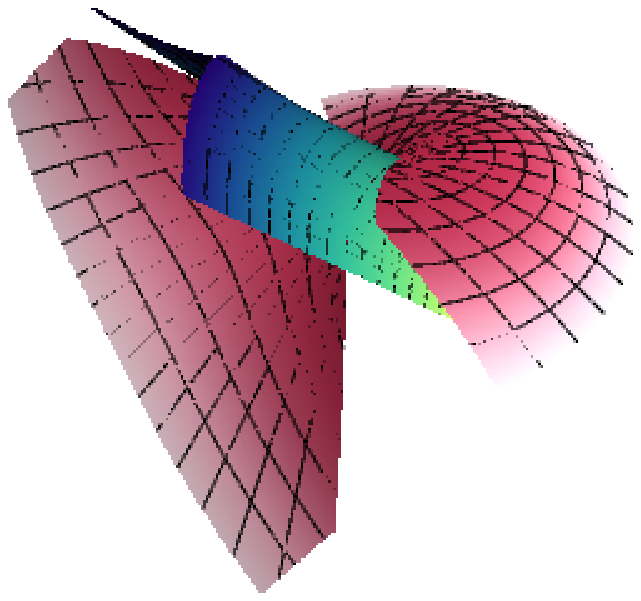}
\caption{Swallowtail given in \eqref{eq:swex} with its
osculating and normal developables}
\label{fig:odfsw22}
\end{figure}
\end{example}

\section{Special swallowtails}
In this section we consider the case when the singular values 
of $OD_f$ and $ND_f$ are special. In particular,
the empty set and a point.
Namely, we study the cases 
$OD_f$ and $ND_f$ are a cylinder or a cone.
Let $f:(\R^2,0)\to (\R^3,0)$ be a frontal and $0$ a singular point
of the second kind.
Let $\{\e,\bb,\bnu\}$ be the Darboux frame 
defined in Subsection \ref{sec:frame}.

We now define the notion of contour edges.
For a unit vector $\bm{k} \in S^2=\{\bx\in\R^3\,|\,|\bx|=1\}$, 
we say that $S(f)$
is the 
{\it  tangential contour edge of the 
orthogonal projection with the direction\/}
$\bm{k}$ if
\[
S(f)=\{ (u,0)\in (\R^2,0) \ |\ \langle \bm{\nu}(u), \bm{k} \rangle = 0\}.
\]
We also say that $S(f)$ is the 
{\it  normal contour edge
 of the orthogonal projection
with the direction\/}
$\bm{k}$ if
\[
S(f)=\{ (u,0)\in (\R^2,0) \ |\ \langle \bm{b}(u), \bm{k} \rangle = 0\}.
\]
Moreover, for a point $\bm{c}\in \R^3,$
we say that $S(f)$ is the 
{\it  tangential contour edge
 of the central projection\/}
(respectively, 
{\it  normal contour edge
 of the central projection\/})
{\it with the center\/} $\bm{c}$ if 
\[
\begin{array}{rcl}
S(f)&=&
\{(u,0)\in (\R^2,0)\ |\ \langle f(u,0) - \bm{c}, \bm{\nu}(u)\rangle = 0\ \}.\\
\big(\text{respectively, }
S(f)&=&
\{(u,0)\in (\R^2,0)\ |\ \langle f(u,0) - \bm{c}, \bm{b}(u)\rangle = 0\ \}.\big)
\end{array}
\]
For a regular surface, the notion of contour edges 
corresponds to the notion of 
contour generators \cite{book4}.
%It is also the singular set of 
%the central projection with the center $\bm{c}.$

\subsection{Osculating developable is a cylinder or a cone}
\begin{theorem} 
With the same notations as the previous sections, 
we have the following\/{\rm :}
\par\noindent
{\rm (A)} Suppose that\/ 
$\tilde\kappa_t^2+\tilde\kappa_\nu^2\not= 0.$ 
Then the following properties are equivalent\/{\rm :}
\begin{enumerate}
\item\label{itm:511} $OD_f$ is a cylinder,
\item\label{itm:512} $\tilde\delta_o \equiv 0$,
\item\label{itm:513} $\bm{\nu}$ is a part of a great circle in\/ $S^2.$
\item\label{itm:514} $S(f)$ is a tangential contour edge 
                    with respect to an orthogonal projection.
\item\label{itm:515} $\overline{D_o}$ is a constant vector.
%\item\label{itm:516} $\overline{\bs_o'}$ is  a constant vector.
\end{enumerate}
\par\noindent
{\rm (B)} Suppose that\/ 
$\tilde\kappa_g^2+\tilde\kappa_t^2\not= 0.$ 
Then the following properties are equivalent\/{\rm :}
\begin{enumerate}
\item\label{itm:521} $ND_f$ is a cylinder,
\item\label{itm:522} $\tilde\delta_n(u)
 \equiv 0$,
\item\label{itm:523} $\bb$ is a part of a great circle in\/ $S^2,$
\item\label{itm:524} $S(f)$ is a normal contour edge 
                      with respect to an orthogonal projection.		
\item\label{itm:525} $\overline{D_n}$ is a constant vector.
%\item\label{itm:526} $\overline{\bs_n'}$ is a constant vector.
\end{enumerate}
\end{theorem}
\begin{proof}
We show the assertions (A).
By \eqref{eq:dop}, we see that the equivalency of 
\eqref{itm:511}, \eqref{itm:512} and \eqref{itm:525}.
The condition $\tilde\kappa_t^2+\tilde\kappa_{\nu}^2\not= 0$ means that 
$\bm{\nu}$ is a non-singular spherical curve.
Moreover, since
$
\bm{\nu}''=(\tilde\kappa_g \tilde\kappa_t - \tilde\kappa_\nu')\be
+(-\tilde\kappa_\nu \tilde\kappa_g -  \tilde\kappa_t')\bb
$ and by \eqref{eq:deltao},
we see that 
$\det(\bnu,\bnu',\bnu'')=\tilde\delta_o$.
This implies that
the geodesic curvature of $\bm{\nu}$
is
$\tilde\delta_o(\tilde\kappa_t^2+\tilde\kappa_{\nu}^2)^{-3/2}$,
and it shows that the equivalency of \eqref{itm:512} 
and \eqref{itm:513}.
We assume \eqref{itm:515}. Then $\overline{D_o}(u)$ is a 
constant vector $\overline{D_o}$.
Thus $\inner{\bnu(u)}{\overline{D_o}}=0$ for any $u$.
This implies that $S(f)$ is a tangential contour edge with respect
to $\overline{D_o}$, and it implies \eqref{itm:514}.
Conversely, we assume \eqref{itm:514}. 
Then there exists a vector $\bm{k}$ such that
$\inner{\bnu(u)}{\bm{k}}=0$ holds for any $u$.
This implies that $\bnu(u)$ belongs to the normal plane of $\bm{k}$
passing through the origin, and it implies \eqref{itm:513}.
Thus the assertion of (A) holds.
One can show the assertion of (B) by the same method to 
the proof of (A) using \eqref{eq:deltab} and \eqref{eq:dnp} 
instead of \eqref{eq:deltao} and \eqref{eq:dop}.
\end{proof}
We also have the following theorem.
\begin{theorem} 
With the same notations as above, we have the following\/{\rm :}
\par\noindent
{\rm (A)} Suppose that\/ 
$\tilde\kappa_t^2+\tilde\kappa_{\nu}^2\not= 0$
and
$\tilde\delta_o\not= 0$
for any $u\in I$.
Then the following properties are equivalent\/{\rm :}
\begin{enumerate}
\item\label{itm:odcone1} $OD_f$ is a cone,
\item\label{itm:odcone2} $\tilde\sigma_o\equiv0$,
%\item\label{itm:odcone3} $\bm{\nu}$ is ? 
\item\label{itm:odcone4} $S(f)$ is a tangential contour edge 
      with respect to a central projection.
\item\label{itm:odcone5} $\bs_o$ is  a constant vector.
\end{enumerate}
\par\noindent
{\rm (B)} Suppose that\/ 
$\tilde\kappa_t^2+\tilde\kappa_g^2\not= 0$
and
$\tilde\delta_n\not= 0$
for any $u\in I$. 
Then the following properties are equivalent\/{\rm :}
\begin{enumerate}
\item\label{itm:ndcone1} $ND_f$ is a cone,
\item\label{itm:ndcone2} $\tilde\sigma_n\equiv0$,
%\item\label{itm:ndcone3} $\bb$ is ? 
\item\label{itm:ndcone4} $S(f)$ is a normal contour edge
with respect to a central projection.
\item\label{itm:ndcone5} $\bs_n$ is  a constant vector.
\end{enumerate}
\end{theorem}
\begin{proof}
By \eqref{eq:sodp}, we see that the equivalency of 
\eqref{itm:ndcone1}, \eqref{itm:ndcone2} and \eqref{itm:ndcone5}.
We assume \eqref{itm:ndcone2}.
Then $\bs_o(u)$ is a constant vector for any $u$.
We set $\bc=\bs_o(u)$.
Then by \eqref{eq:stro}, $f(u,0)-\bc$ is parallel to
$\overline{D_o}(u)$.
Thus 
$\inner{f(u,0)-\bc}{\bnu(u)}
=
\inner{\overline{D_o}(u)}{\bnu(u)}=0$ holds for any $u$.
This implies \eqref{itm:ndcone4}.
Conversely, we assume \eqref{itm:ndcone4}. 
Then there exists a vector $\bc$ such that
$\inner{f(u,0)-\bc}{\bnu(u)}\equiv0$.
By \eqref{eq:stro}, $\bs_o(u)-f(u,0)$
is parallel to $\overline{D_o}(u)$,
$\inner{\bs_o(u)-\bc}{\bnu(u)}\equiv0$.
Differentiating this equation by $u$,
and noticing $\inner{\bs_o'(u)}{\bnu(u)}\equiv0$
by \eqref{eq:sodp},
we have $\inner{\bs_o(u)}{\bnu'(u)}\equiv0$.
On the other hand, by \eqref{eq:sodp} and \eqref{eq:frenet},
we see that $\inner{\bs_o'(u)}{\bnu'(u)}\equiv0$.
Thus
differentiating $\inner{\bs_o(u)}{\bnu'(u)}\equiv0$ by $u$,
we have $\inner{\bs_o(u)}{\bnu''(u)}\equiv0$.
On the other hand, by \eqref{eq:frenet}, the three vectors
$\bnu(u),\bnu'(u),\bnu''(u)$ are linearly independent if
and only if $\tilde\delta_o(u)\ne0$.
Hence 
$$
\inner{\bs_o(u)-\bc}{\bnu(u)}\equiv
\inner{\bs_o(u)-\bc}{\bnu'(u)}\equiv
\inner{\bs_o(u)-\bc}{\bnu''(u)}\equiv0
$$
implies $\bs_o(u)-\bc\equiv0$, and this 
implies \eqref{itm:ndcone1}.
Thus the assertion of (A) holds.
One can show the assertion of (B) by the same method to 
the proof of (A) using \eqref{eq:sndp}
instead of \eqref{eq:sodp}.
\end{proof}

Let us consider 
a cylinder $c_y:(\R^2,0)\to(\R^3,0)$ and
a cone $c_o:(\R^2,0)\to(\R^3,0)$, and consider
and a Whitney cusp $f_w:(\R^2,0)\to(\R^2,0)$.
Here, a Whitney cusp is a map germ
which is $\A$-equivalent to $(u,v)\mapsto(u,v^3+uv)$.
Then each $c_y\circ f_w:(\R^2,0)\to(\R^3,0)$ and
$c_o\circ f_w:(\R^2,0)\to(\R^3,0)$ is a frontal
and each $0$ is a singularity of the second kind.
Thus $OD_{c_y\circ f_w}$ is a cylinder, and
$OD_{c_o\circ f_w}$ is a cone.
These examples are a kind of trivial examples.
Here we give non-trivial examples.
\begin{example}\label{ex:swodfcyl}
Let us set
$$f(u,v)=
\left(-\dfrac{u^2}{2}+v,
\dfrac{u^3}{3}-u v,
\dfrac{u^4}{8}+\dfrac{1}{2} \left(\dfrac{u^2}{2}-v\right)^2-\dfrac{u^2 v}{2}
\right).
$$
Then we see that $S(f)=\{v=0\}$ and $\tilde\delta_o(u)=0$.
Thus $OD_f$ is a cylinder.
The figure of this example is given in Figure \ref{fig:swodfcyl}.
\begin{figure}[!ht]
\centering
  \includegraphics[width=.2\linewidth]{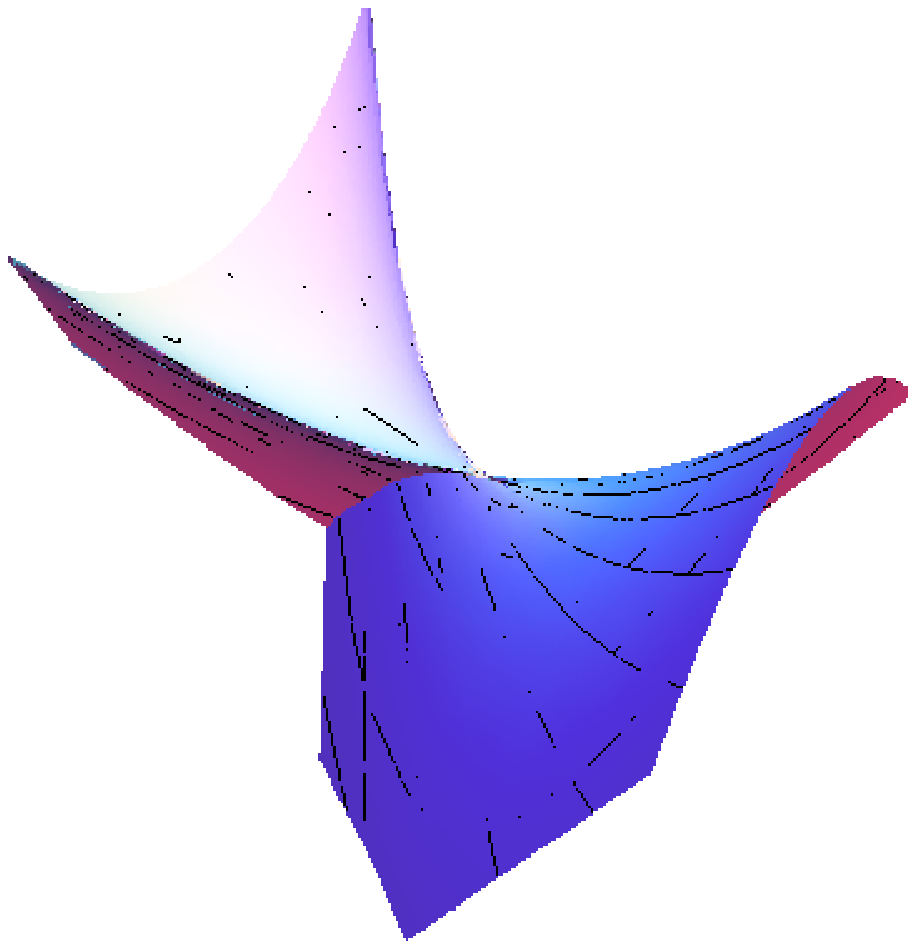}
 \hspace{2mm}
  \includegraphics[width=.2\linewidth]{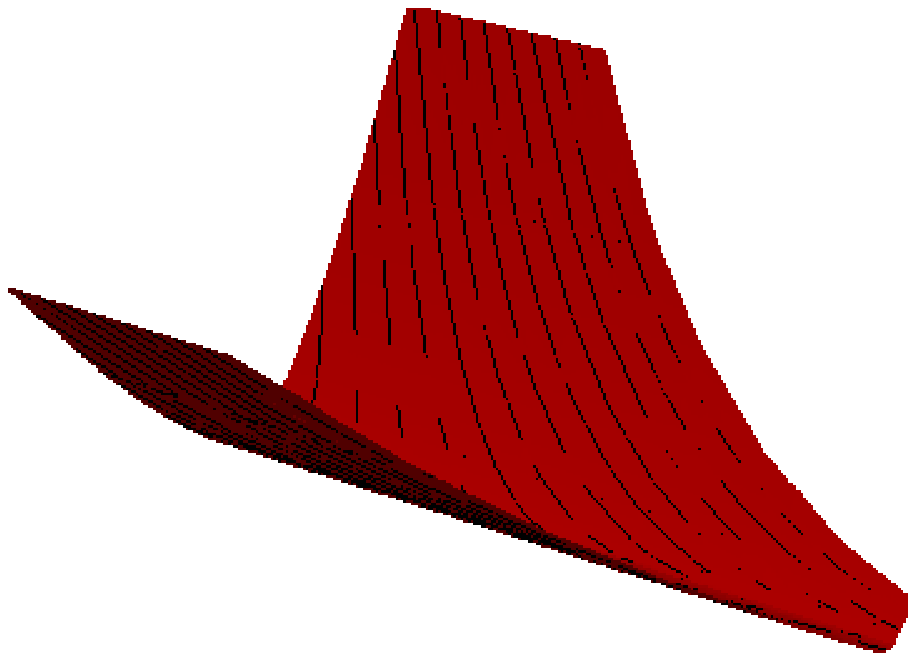}
 \hspace{2mm}
  \includegraphics[width=.2\linewidth]{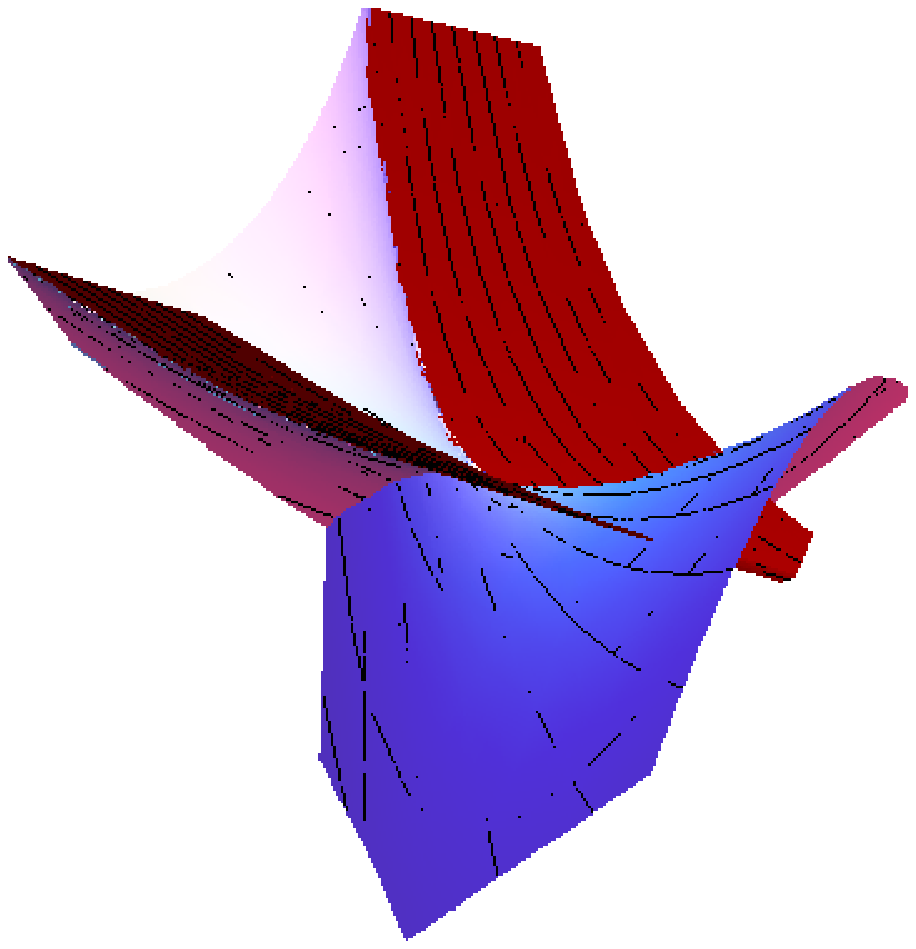}
\caption{Swallowtail in Example \ref{ex:swodfcyl}, 
with its osculating developable and both surfaces}
\label{fig:swodfcyl}
\end{figure}
\end{example}

Let us consider 
a plane $p:(\R^2,0)\to(\R^3,0)$ and
a sphere $s:(\R^2,0)\to(\R^3,0)$, and consider
and a Whitney cusp $f_w:(\R^2,0)\to(\R^2,0)$.
Then each $p\circ f_w:(\R^2,0)\to(\R^3,0)$ and
$s\circ f_w:(\R^2,0)\to(\R^3,0)$ is a frontal
and each $0$ is a singularity of the second kind.
Thus $ND_{c_y\circ f_w}$ is a cylinder, and
$ND_{c_o\circ f_w}$ is a cone. 
We say that $f:(\R^2,0)\to(\R^3,0)$ is a {\it Whitney frontal} 
if it is $\mathcal{A}$-equivalent to $(u,v)\mapsto(u,v^2,0)$ 
or $(u,v)\mapsto(u,v^3+uv,0)$. 
Then $p\circ f_w$ and $s\circ f_w$ are Whitney frontals. 
These examples are a kind of trivial examples.
Here we give non-trivial examples
\begin{example}\label{ex:swndfcyl}
Let us set
$$f(u,v)=
\left(
-\dfrac{u^2}{2}+v,
\dfrac{u^3}{3}-u v,
\dfrac{u^4}{8}-\dfrac{u^2 v}{2}\right).
$$
Then we see that $S(f)=\{v=0\}$ and $\tilde\delta_n(u)=0$.
Thus $ND_f$ is a cylinder.
The figure of this example is given in Figure \ref{fig:swndfcyl}.
\begin{figure}[!ht]
\centering
  \includegraphics[width=.2\linewidth]{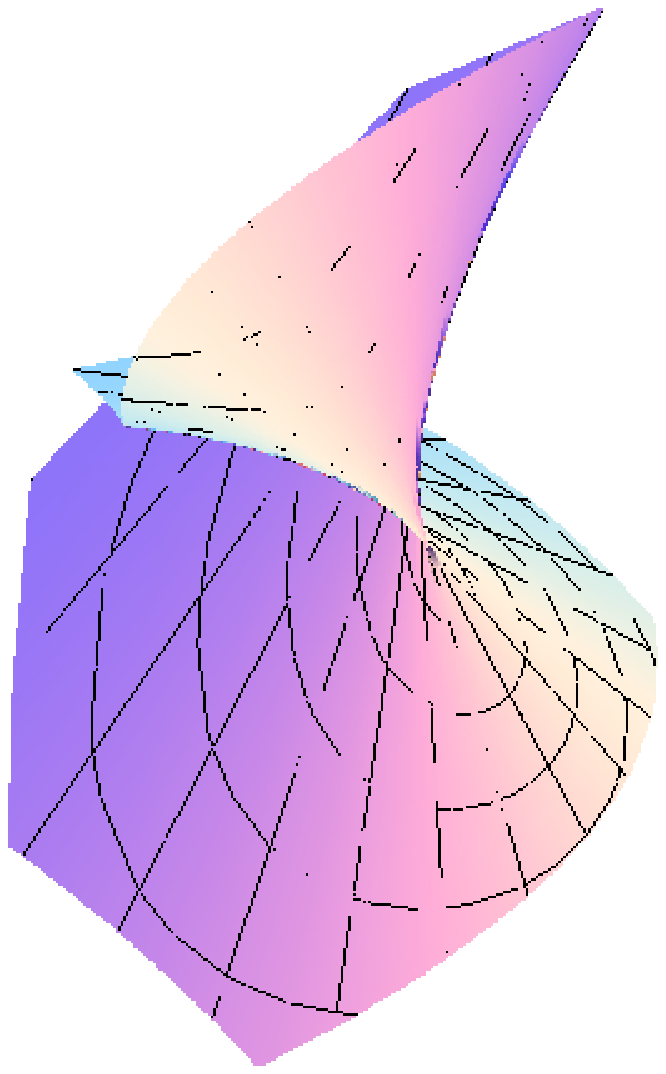}
 \hspace{2mm}
  \includegraphics[width=.2\linewidth]{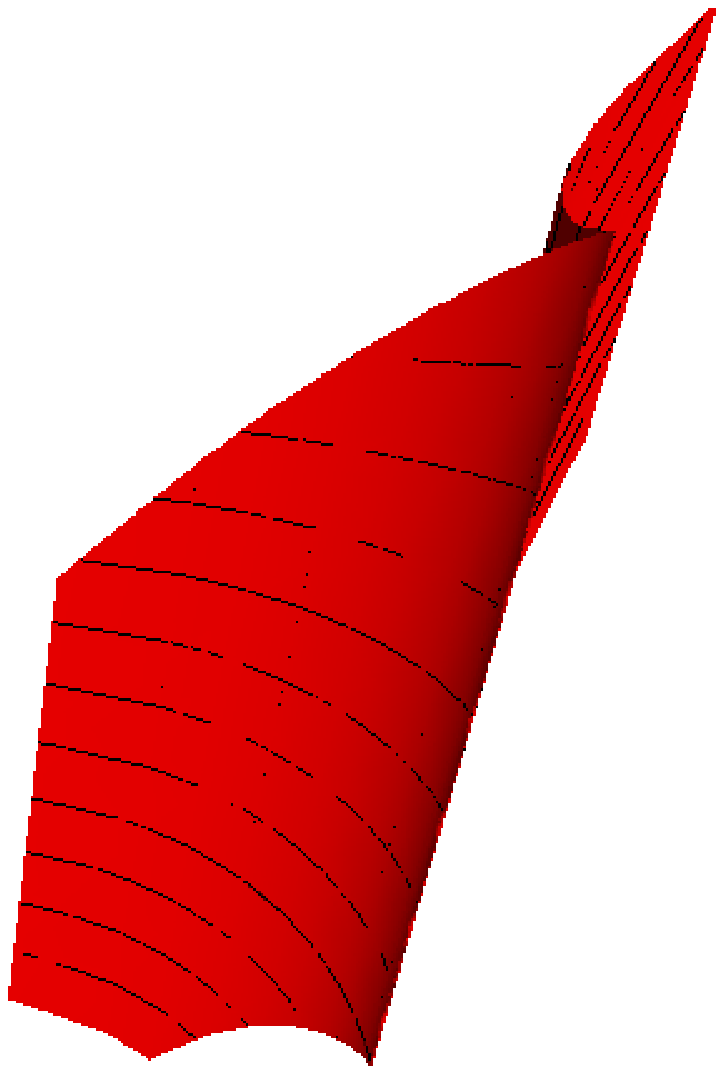}
 \hspace{2mm}
  \includegraphics[width=.2\linewidth]{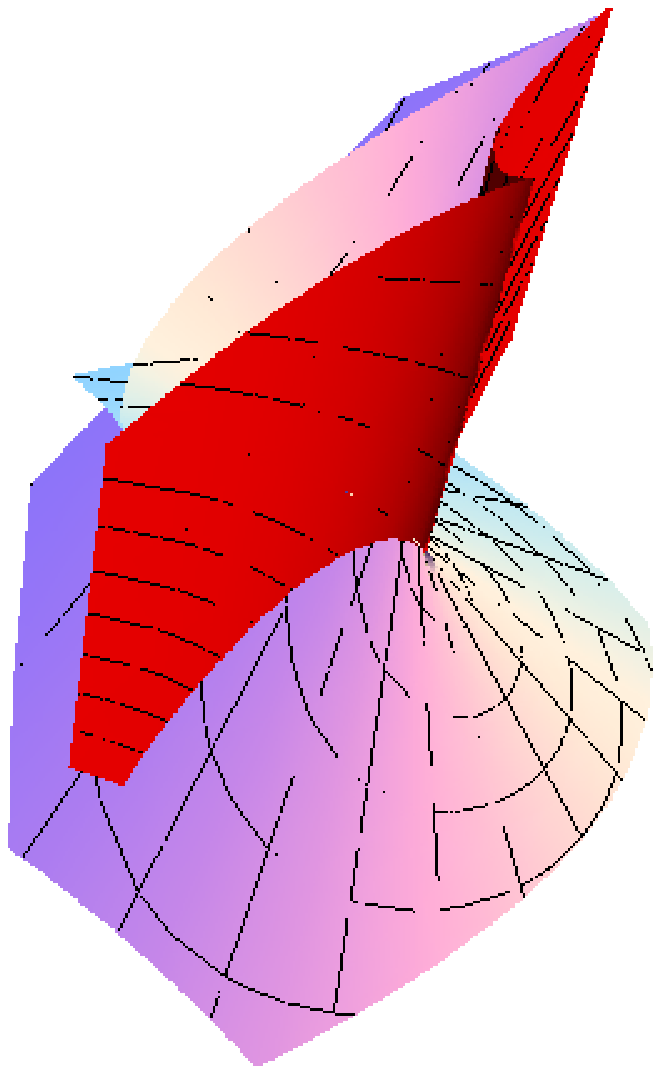}
\caption{Swallowtail in Example \ref{ex:swndfcyl}, 
with its normal developable and both surfaces}
\label{fig:swndfcyl}
\end{figure}
\end{example}
We also give another example of a singular point
whose normal developable
is a cylinder.
It is not a swallowtail but a singular point of the second kind.
\begin{example}\label{ex:ndfcyl}
Let us set
$$
f(u,v)=
(u^2-v,-u^3+3 u (u^2-v),(2 u^2-v)^3 v^3).
$$
Then we see that $f$ is a frontal and 
$0$ is a singular point of the second kind but is not a swallowtail.
We also see that $S(f)=\{v=0\}$ and $\tilde\delta_n(u)=0$.
Thus $ND_f$ is a cylinder.
The figure of this example is given in Figure \ref{fig:ndfcyl}. 
In this case, the image is locally homeomorphic to the swallowtail. 
We say that $f:(\R^2,0)\to(\R^3,0)$ is a {\it quasi-swallowtail} 
if it is a frontal, $0$ is a singular point of the second kind and 
the image is homeomorphic to the swallowtail. 
Then this example is the quasi-swallowtail.
\begin{figure}[!ht]
\centering
  \includegraphics[width=.3\linewidth]{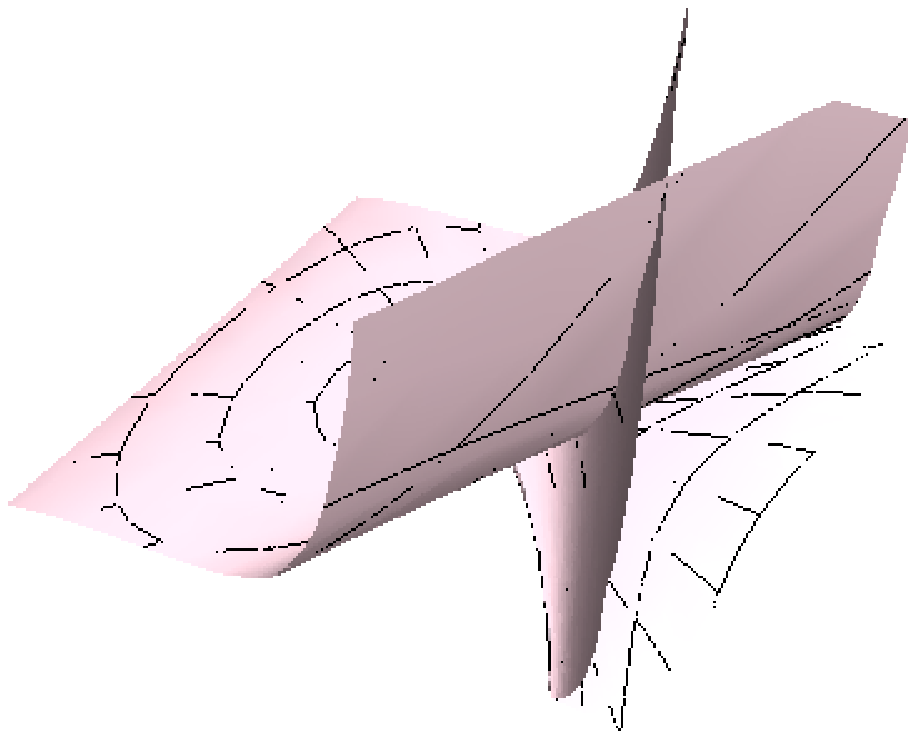}
 \hspace{2mm}
  \includegraphics[width=.1\linewidth]{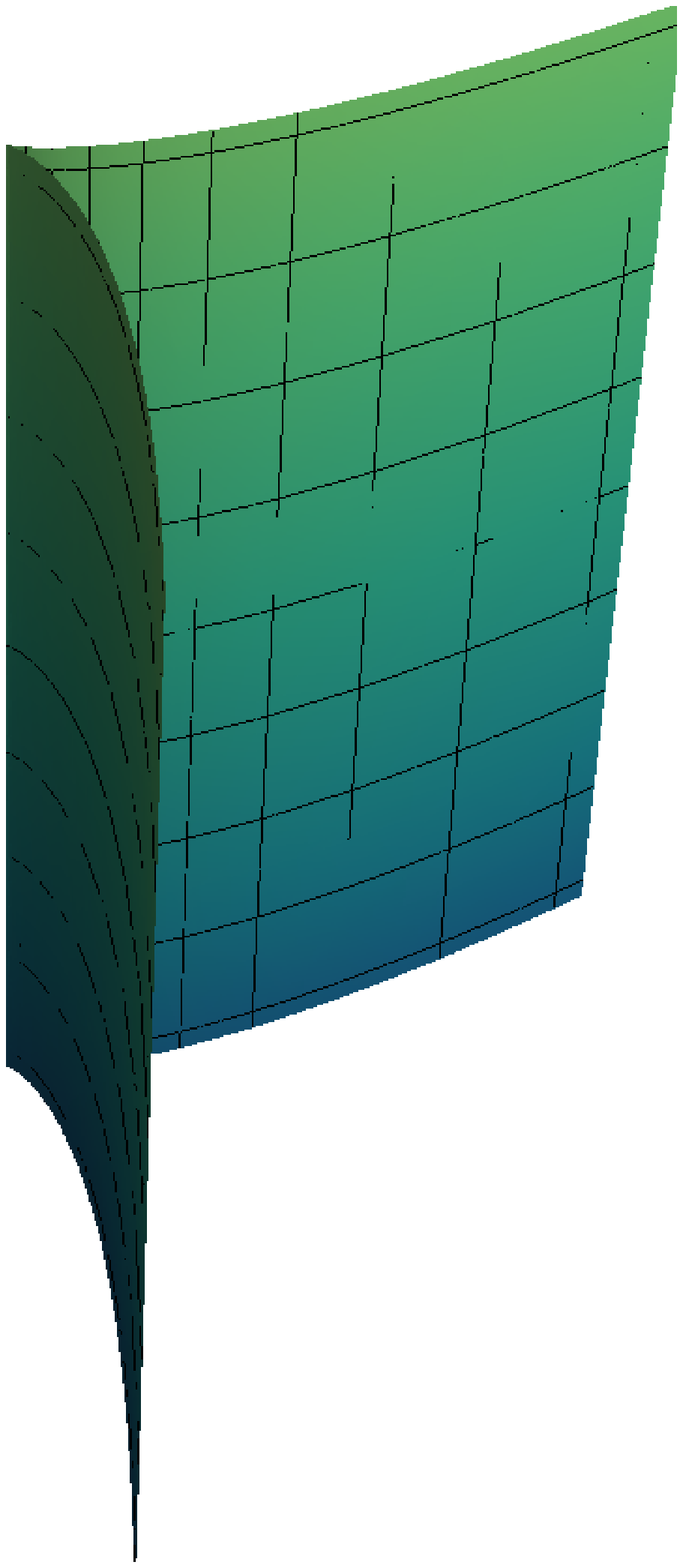}
 \hspace{2mm}
  \includegraphics[width=.3\linewidth]{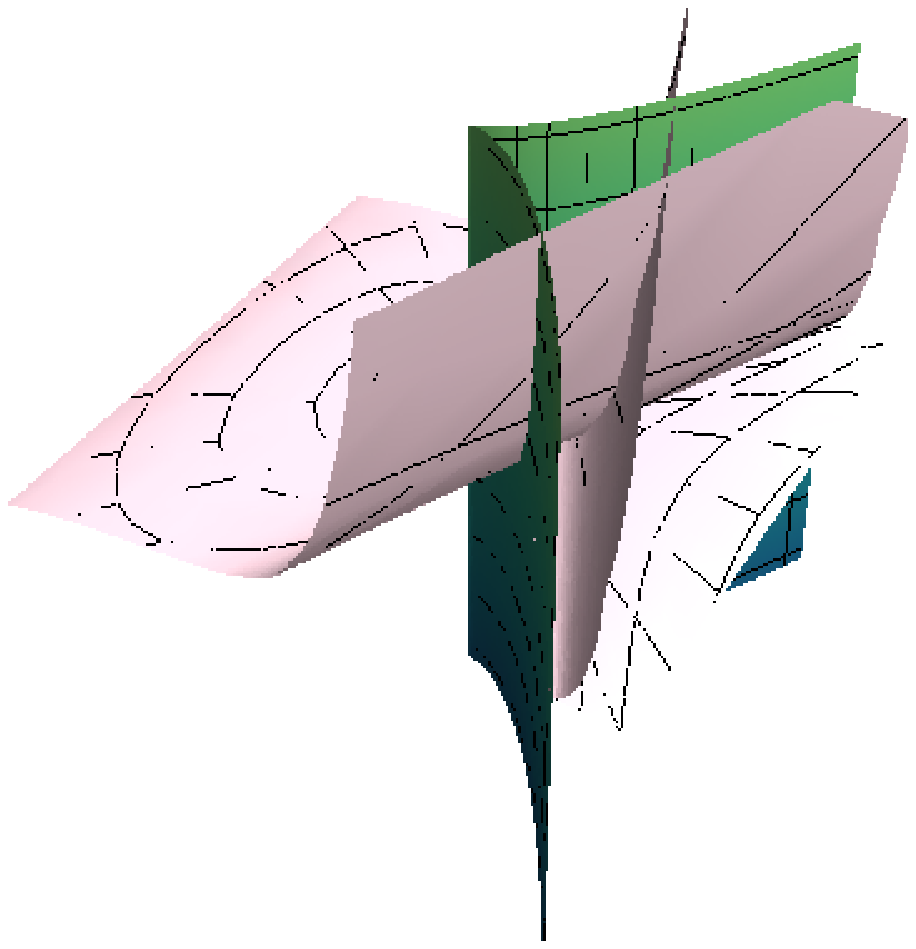}
\caption{Surface in Example \ref{ex:ndfcyl}, 
with its normal developable and both surfaces}
\label{fig:ndfcyl}
\end{figure}
\end{example}

%\section{thebibliography}

\begin{flushright}
\begin{tabular}{ll}
\begin{tabular}{l}
(Izumiya)\\
Department of Mathematics, \\
Hokkaido University,\\
Sapporo 060-0810, Japan\\
{\tt izumiyaO\!\!\!amath.sci.hokudai.ac.jp}\\
\phantom{a}\\
\phantom{a}
\end{tabular}
&
\begin{tabular}{l}
Department of Mathematics, \\
Kobe University,
Rokko 1-1, Nada, \\
Kobe 657-8501, Japan\\
(Saji)\\
{\tt sajiO\!\!\!amath.kobe-u.ac.jp}\\
(Teramoto)\\
{\tt teramotoO\!\!\!amath.kobe-u.ac.jp}
\end{tabular}
\end{tabular}
\end{flushright}

\end{document}